\documentclass[11pt,a4paper,twoside,final,biblatex]{scrartcl}
\usepackage{a4wide}
\usepackage{amsfonts}
\usepackage{amsmath}
\usepackage{amsthm}
\usepackage{mathtools}
\usepackage{amssymb}
\usepackage{array}
\usepackage{multicol}
\usepackage{wrapfig}
\usepackage{dsfont}
\usepackage[utf8]{inputenc}
\usepackage[T1]{fontenc}
\usepackage{caption}
\usepackage{subcaption}
\usepackage{enumerate} 
\usepackage{xifthen}
\usepackage[urlcolor=blue,colorlinks=false]{hyperref}
\usepackage{algorithm}
\usepackage{graphicx}
\usepackage{color}
\usepackage{pgf}

\usepackage{multirow}
\usepackage{multicol}
\usepackage{tabularx}
\usepackage{diagbox}
\usepackage{longtable}
\usepackage{listings} 
\usepackage{cleveref}
\usepackage{tikz}
\usetikzlibrary{math}
\usetikzlibrary{patterns}
\usepackage{graphicx}
\usepackage{wrapfig}
\usepackage{changes} 
\usepackage[export]{adjustbox}

\newcommand{\N}{\ensuremath{\mathbb{N}}}

\newcommand{\NZ}{\ensuremath{\mathbb{N}_{0}}}
\newcommand{\T}{\ensuremath{\mathbb{T}}}
\renewcommand{\S}{\ensuremath{\mathbb{S}}}

\newcommand{\Z}{\ensuremath{\mathbb{Z}}}

\newcommand{\R}{\ensuremath{\mathbb{R}}}
\newcommand{\C}{\ensuremath{\mathbb{C}}}

\newcommand{\ii}{\mathit{i}}
\newcommand{\e}{\textnormal{e}}

\newcommand{\eip}[1]{\textnormal{e}^{2\pi\ii{#1}}}

\renewcommand{\d}{\,\mathrm{d}}

\newcommand{\supp}{\operatorname{supp}}
\newcommand{\disp}{\operatorname{disp}}

\usepackage{algpseudocode}
\usepackage{algorithm}

\newtheorem{theorem}{Theorem}[section]
\newtheorem{lemma}[theorem]{Lemma}
\newtheorem{corollary}[theorem]{Corollary}
\newtheorem{proposition}[theorem]{Proposition}
\theoremstyle{definition}
\newtheorem{remark}[theorem]{Remark}
\newtheorem{definition}[theorem]{Definition}
\newtheorem{example}[theorem]{Example}

\numberwithin{equation}{section}
\numberwithin{table}{section}
\numberwithin{figure}{section}

\newcommand{\bend}{\hspace*{0ex} \hfill \hbox{\vrule height
    1.5ex\vbox{\hrule width 1.4ex \vskip 1.4ex\hrule  width 1.4ex}\vrule
    height 1.5ex}}

\long\def\symbolfootnote[#1]#2{\begingroup%
\def\thefootnote{\fnsymbol{footnote}}\footnote[#1]{#2}\endgroup}

\crefname{lemma}{Lemma}{Lemmata}
\crefname{definition}{Definition}{Definitions}
\crefname{theorem}{Theorem}{Theorems}
\crefname{corollary}{Corollary}{Corollaries}
\crefname{equation}{}{}
\crefname{remark}{Remark}{Remarks}
\crefname{algorithm}{Algorithm}{Algorithms}
\crefname{chapter}{Chapter}{Chapters}
\crefname{section}{Section}{Sections}
\crefname{table}{Table}{Tables}
\crefname{figure}{Figure}{Figures}
\crefname{example}{Example}{Examples}
\crefname{appendix}{Appendix}{Appendices}

\renewcommand{\thefootnote}{\fnsymbol{footnote}}

\allowdisplaybreaks

\title{Positive quadrature and mobile sampling of multivariate trigonometric polynomials}

\author{Stefan Kunis\footnote{University Osnabrueck, Germany. Email: \url{stefan.kunis@uos.de}}}

\begin{document}

\date{}

\maketitle

\begin{abstract}
The geometric properties of quadrature points in one and multiple dimensions are a classical topic in numerical analysis. Recently, generalized quadrature methods, where discrete points and weights are replaced by integration along curves, have attracted growing interest.
This paper focuses on positive generalized quadratures that are exact for multivariate trigonometric polynomials.
We derive upper bounds on the covering radius of such quadratures using sign-localized test functions, a technique that also extends naturally to algebraic polynomials on intervals, spherical polynomials, and hyperbolic cross trigonometric polynomials.
For lower bounds on the length of quadrature curves, we compare two recent approaches in the context of multivariate trigonometric polynomials.
The second part of the paper studies a well-known, geometrically simple curve that underlies rank-1 Korobov lattice rules for periodic functions. We prove that these curves are quasi-optimal with respect to multiple optimality criteria and demonstrate its adaptability to the 2-sphere.

   
\end{abstract}

\section{Introduction}

Quadrature rules are a fundamental object of applied mathematics and stable sampling is an inherently connected topic of signal processing.
Classical schemes in both areas traditionally rely on point evaluations. However, dense data acquisition along curves has recently been established as a model in diverse fields such as magnetic resonance and particle imaging as well as drone and satellite scanning.
For bandlimited functions on $\R^d$, Unnikrishnan and Vetterli coined the term 'mobile sampling' in \cite{UnVe13}, with follow-up papers including \cite{JaNeRo21,RaUlZl23,JaMiVe24}, among others.
The polynomial exactness of generalized quadrature through integration along curves on the $d$-dimensional sphere and cube has been prominently studied in \cite{EhGr23,EhGrKa25,Li24,EhGr26,KiMi26} and \cite{BoDeVi17}, respectively.
Beyond measuring the density of a set in a domain, curves naturally possess length, and both quantities can be related to the degree of exactness and bandwidth.

For finite point sets (and more general sets), there is a relation between the density of points and their suitability for sampling or quadrature.
Motivated by the famous Shannon sampling theorem and strongly oversimplifying, there is \emph{a sufficient condition:} if the covering radius (largest hole) of a set is of size at most $c/n$, then any $n$-bandlimited function can be reconstructed from these samples, \emph{and a necessary condition:} if any $n$-bandlimited function can be reconstructed from samples, then the covering radius of these sampling points is at most $C/n$.

In the context of spherical harmonics, a nice discussion on the asymptotic notion of upper Beurling density can be found, e.g., in \cite{Ma07,MaPr14}.
Another line of density theorems considers interpolation problems (again see the same references for spherical harmonics), dating back to Ingham and recently experienced a renaissance motivated by super-resolution microscopy, see, e.g., \cite{Mo15,KuNaSt22,HoKu23}.
Here, sufficient conditions on the \emph{separation} of points allowing for interpolation have been proven via the Poisson summation formula using so-called Beurling--Selberg minorant functions, see, e.g., \cite[Thm.~3.8]{HoKu23}.
Recently, such functions have been named \emph{sign-localized} in the context of analytic number theory, see, e.g., \cite{CoDoGo24} and references therein.
Notably, sign-localized zonal polynomials were also employed by Reimer and Yudin for a necessary condition for positive quadrature on the sphere already in \cite[Thm.~6.21]{Re03}.

The purpose of the present paper is to prove upper bounds on the covering radius and lower bounds on the length of a curve if a positive generalized quadrature is exact for multivariate trigonometric polynomials. This is supplemented by a particularly simple curve which integrates multivariate trigonometric polynomials exactly and meets the geometric bounds up to an explicit, dimension-dependent constant. We mainly follow the approaches developed in \cite{Re03,BoDeVi17,EhGr23,EhGrKa25,EhGr26} on the sphere and the unit cube, use them in the geometrically simpler situation of the $d$-dimensional torus, and clarify their relations. In detail, our contributions are as follows:
\begin{itemize}
    \item Given any positive generalized quadrature, we establish explicit upper bounds on the covering radius with respect to the inverse degree of exactness in \cref{thm:deltaT}.
    We follow the proof by contradiction \cite[Chapter 6]{Re03}. Our main contribution is the construction of appropriate test functions which are band limited and have a specific sign change in spatial domain. This is in contrast to compactly supported test functions and approximation rate based estimates in \cite{EhGr23,EhGrKa25}. Via Tschakaloff's theorem it suffices to consider finite point sets here. 
    \item We give generalizations to algebraic polynomials on intervals and spherical polynomials by using explicit estimates on the zeros of orthogonal polynomials. Another generalization allows to bound the dispersion of a point set for positive quadrature of hyperbolic cross trigonometric polynomials, see \cref{sec:gen}.
    \item Following two approaches in \cite{EhGr23,EhGrKa25} and \cite{EhGr26}, we prove lower bounds on the length of a curve allowing for positive generalized quadrature in \cref{cor:Tschak2} and \cref{thm:Improv}. The first approach uses a purely geometric argument relating the covering radius and the length of a curve, also well known in the geometric traveling salesman problem \cite{BaClDu24}. The latter cleverly bounds the length of the curve segment when intersected with a box of appropriate small size \cite{EhGr26}.
    \item In \cref{sec:simple}, we consider a specific periodic straight line which has been used in different contexts already. This curve $\gamma$ allows for exact integration of multivariate trigonometric polynomials $f$ since the composition $f\circ\gamma$ is a univariate trigonometric polynomial, see \cref{thm:gamma1}.
    The very same construction for algebraic polynomials on the unit cube can be found in \cite{BoDeVi17}. We compute the exact covering radius and length of the curve in \cref{thm:gamma2}, which yields that the before proven necessary bounds are optimal up to a dimension dependent constant.
    \item Finally, we transfer this curve from the two-dimensional torus to the two-dimensional sphere in \cref{sec:sphere}, which yields a weighted version of \cite[Thm.~1.2]{EhGr23}. Also here, actual covering radius and length as well as the necessary bounds differ only by an explicit dimension dependent constant. The guiding principle being that a spherical harmonic in standard parametrization is a bivariate trigonometric polynomial with additional symmetry.
\end{itemize}

\section{Preliminaries}

We let $\T^d\simeq [0,1)^d\simeq [-\frac12,\frac12)^d$ denote the $d$-dimensional torus (unit-cube with opposing sides identified) and consider multivariate trigonometric polynomials $f\colon\T^d\to\C$,
\begin{align*}
    f(x)=\sum_{\|k\|_\infty\le n} c_k \eip{kx}
    =\sum_{k_1=-n}^n\dots\sum_{k_d=-n}^n c_{k_1,\hdots,k_d} \eip{k_1x_1+k_2x_2+\hdots+k_dx_d},
\end{align*}
denoted by $f\in \Pi_n^d$.
For an arbitrary set $Y\subset\T^d\simeq[0,1)^d$, we call
\begin{align*}
  Q[f]=\int_Y f(x)\d \sigma(x)  
\end{align*}
a \emph{generalized quadrature}. We note in passing that the measure $\sigma$ might include an additional non-negative weight. We call this quadrature \emph{exact of degree $n$} if
\begin{align*}
   \int_Y f(x)\d \sigma(x) = \int_{\T^d} f(x)\d x
\end{align*}
for all $f\in \Pi_n^d$.

We mainly consider classical quadrature, where $Y$ is a finite set of points, and integration along a curve. Two geometric quantities which describe the complexity of such sets are the covering radius, defined for all sets, and the length of a curve.
The covering radius (with respect to the wrap-around infinity norm) of a set $Y$ is given by
 \begin{align*}
     \delta=\delta_Y=\sup_{x\in\T^d} \inf_{y\in Y} |y-x|_\infty,\qquad
     |x-y|_\infty=\inf_{z\in\Z^d}\|x-y+z\|_\infty.
 \end{align*}
Now, let $\gamma\colon[0,1]\to\T^d\simeq [0,1)^d$ be a continuous, piecewise differentiable closed curve with at most finitely many self-intersections. We denote its trajectory and length by
 \begin{align*}
    \Gamma=\gamma([0,1]),\qquad
    \ell(\gamma)=\int_{\gamma}\d\sigma(x):=\int_0^1 \|\dot\gamma(t)\|_2\d t,
 \end{align*}
 respectively.

\section{Necessary bounds on covering radius and length}
In this section, we prove an upper bound on the covering radius for positive quadrature rules by constructing a sign-localized trigonometric polynomial with positive integral as test function.
If the covering radius would be too large, we can translate this test function such that the quadrature rule samples non-positive values only and hence cannot be exact.
The test functions as well as their Fourier coefficients are illustrated in Figures \ref{fig:1d} and \ref{fig:2d} for $d=1$ and $d=2$, respectively. Their properties rely on the following \cref{lem:DFn}, which collects standard results for the multivariate versions of the Dirichlet and Fejer kernels.
As a corollary, we also show by \cref{lem:deltaL} and a simple argument relying on Tschakaloff's theorem, see, e.g., \cite[Thm.~1.24]{Sc17}, that the length of a curve in a generalized quadrature is lower bounded.
In the last part of this section, we transfer the result to algebraic polynomials on the interval $[-1,1]$, to spherical harmonics on the sphere $\S^d$, and to a dispersion bound for hyperbolic cross trigonometric polynomials.

\begin{figure}[h!]
 \includegraphics[width=0.99\textwidth]{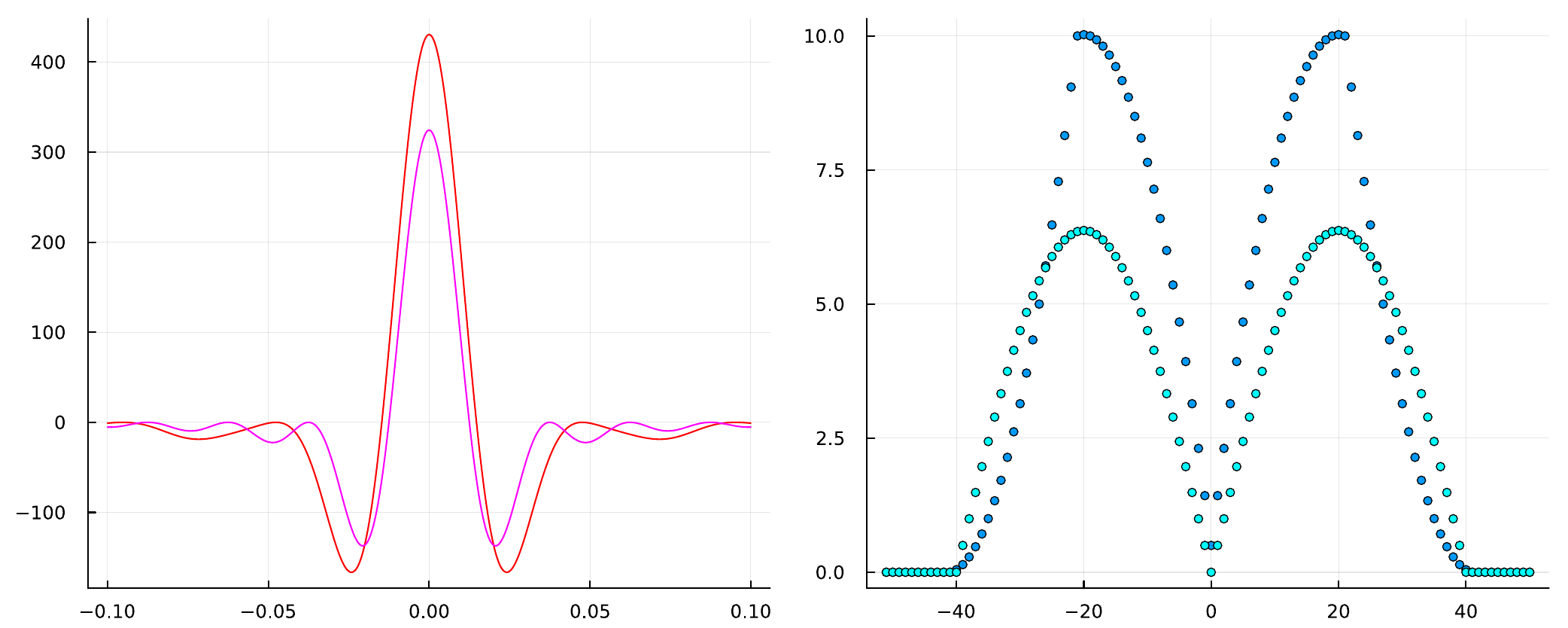}
 \caption{Test function $g(x)=(D_n(x)- (n+\frac{1}{2}))F_n(x)$ for $n=20$ and $d=1$ (left, red) and its Fourier coefficients (right, blue). The test function of Reimer and Yudin, see \cref{rem:Improv}, as well as its Fourier coefficients are shown in magenta and cyan, respectively.}\label{fig:1d}
\end{figure}

\begin{figure}[h!]
 \includegraphics[width=0.99\textwidth]{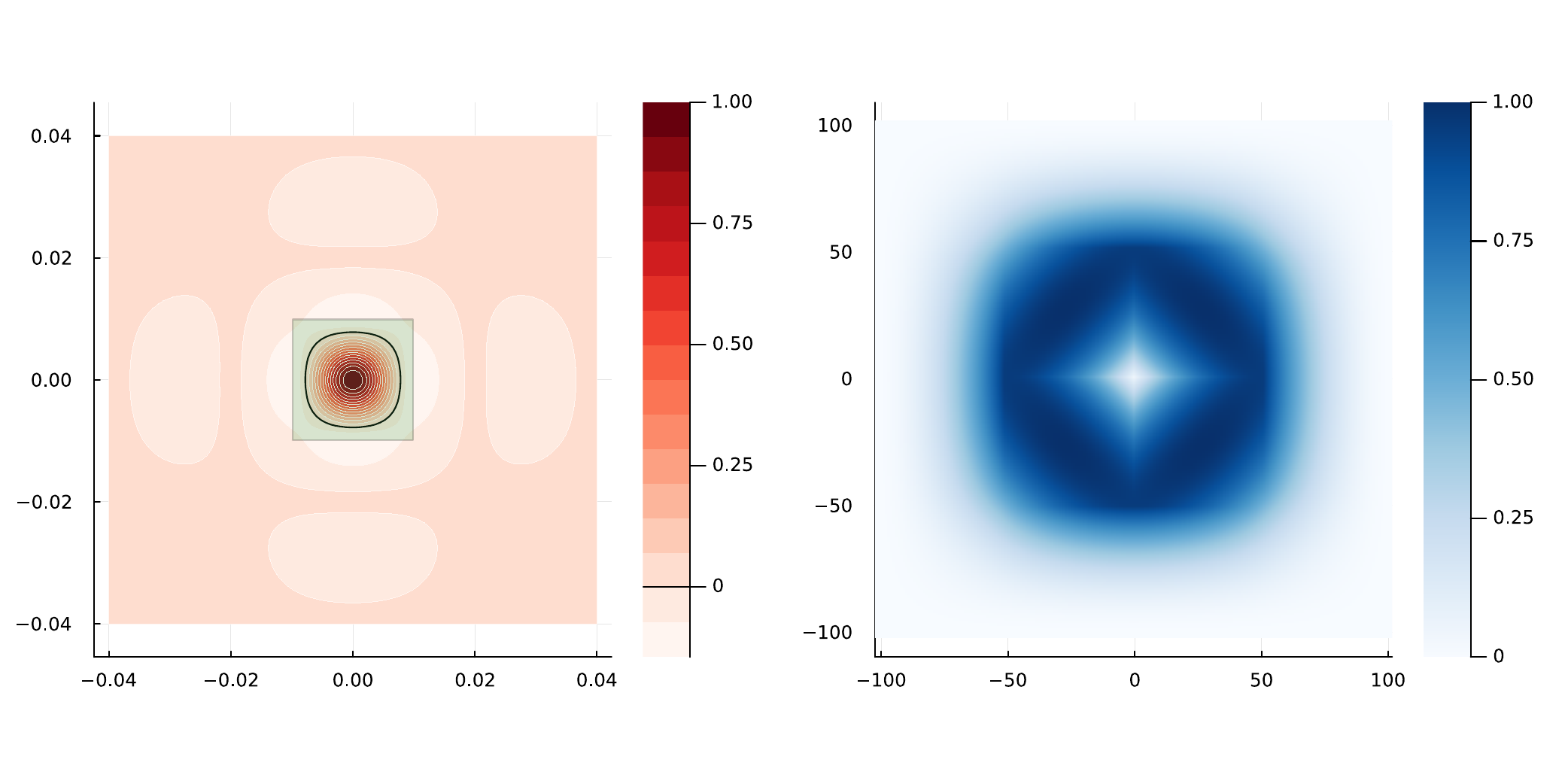}
 \caption{Test function $g(x)=(D_n(x)- (n+\frac{1}{2})^2)F_n(x)$ for $n=50$ and $d=2$ (left, red) and its Fourier coefficients (right, blue). The zero crossing of $g$ is shown as black line, the shaded region has side-length $2/(2n+1)$.}\label{fig:2d}
\end{figure}

\subsection{Upper bounds on the covering radius}
We start in \cref{lem:DFn} by some standard identities and estimates which are included to keep the paper self contained.
\cref{thm:deltaT} gives an upper bound on the covering radius by constructing a simple sign-localized trigonometric polynomial with positive integral and vanishing quadrature sum.
We also include some remarks on possible improvements and variants.
\begin{lemma}\label{lem:DFn}
    Let $d,n\in\N$, then the Dirichlet and Fejer kernel, given by
    \begin{align*}
        D_n(x)&=\begin{cases}\frac{\sin(2n+1)\pi x}{\sin\pi x}, & x\ne 0,\\
                2n+1, & x=0,\end{cases}\qquad
        &D_n(x_1,\hdots,x_d)&=D_n(x_1)\cdot\hdots\cdot D_n(x_d),\\
        F_n(x)&=\begin{cases}
           \frac{1}{n+1}\left(\frac{\sin(n+1)\pi x}{\sin\pi x}\right)^2, & x\ne 0,\\
           n+1, & x=0,\end{cases}\qquad
        &F_n(x_1,\hdots,x_d)&=F_n(x_1)\cdot\hdots\cdot F_n(x_d),
    \end{align*}
    fulfil $D_n,F_n\in\Pi_n^d$ as well as for any $C\in (0,n)$ and $|x|_\infty\ge C/(2n+1)$ the bound
    \begin{align*}
     D_n(x)\le \frac{(2n+1)^{d}}{2C},
    \end{align*}
    and the identities
    \begin{align*}
      \int_{\T^d} D_n(x) F_n(x) \d x=(n+1)^d, \qquad \int_{\T^d} F_n(x)\d x = 1.
    \end{align*}
\end{lemma}
\begin{proof}
 Without loss of generality, let $x_1=|x|_{\infty}$ and use $\sin(2n+1)\pi x_1\le 1$ and $\sin\pi x_1\ge 2x_1\ge 2C/(2n+1)$ in
 \begin{align*}
  D_n(x)\le (2n+1)^{d-1}\frac{\sin(2n+1)\pi x_1}{\sin\pi x_1}
        \le \frac{(2n+1)^{d-1}}{2 x_1}
        \le \frac{(2n+1)^d}{2 C}.
    \end{align*}
 Now note that it suffices to consider the integrals for $d=1$.
 The first identity for the Fejer kernel follows by considering the integral as convolution of a trigonometric polynomial (the Fejer kernel) with the reproducing kernel of the space (the Dirichlet kernel), this evaluates the integral to $F_n(0)$.
 The second integral is simply the zeroth Fourier coefficient of the Fejer kernel - which is known to be one.
\end{proof}

\begin{remark}\label{rem:DFn}
 The bound on the Dirichlet kernel can be slightly improved:
 \begin{enumerate}
     \item For $d=1$ and $C=1$, we bound the main-lobe as well: The so-called 'full width half maximum' in the signal processing community asks for $x$ such that $D_n(x)= (2n+1)/2$ holds. We make the ansatz
     \begin{align*}
       D_n\left(\frac{c}{2n+1}\right)=\frac{2n+1}2
     \end{align*}
     which allows to deduce
     \begin{align*}
         2 \sin c\pi = (2n+1)\sin \frac{c\pi}{2n+1} \overset{n\to\infty}\longrightarrow c\pi
     \end{align*}
     and thus $c\approx 0.603$. In summary, there is $n_0\in\N$ with $D_n(x)\le (2n+1)/2$ for all $n\ge n_0$ and $0.61/(2n+1)\le |x|\le 1/2$.
     \item For $d=2$ and $C=2$ the bound remains valid also for $|x|_\infty\ge 1/(2n+1)$ and this can be seen as follows: If both $|x_1|\ge 1/(2n+1)$ and $|x_2|\ge 1/(2n+1)$, then we can use the bound from the above proof for both factors. If $|x_1|\in [1/(2n+1),2/(2n+1))$ and $|x_2|< 1/(2n+1)$, or vice versa, then $D_n(x)\le 0$ and the bound is trivially fulfilled.
 \end{enumerate}
\end{remark}

\begin{theorem}\label{thm:deltaT}
 If a quadrature rule $Q[f]=\sum_j w_j f(x_j)$ with positive weights $w_j>0$ is exact for $f\in\Pi_{2n}^d$, then the covering radius of the quadrature points is bounded from above by
 \begin{align*}
     \delta=\max_{x\in\T^d} \min_{j} |x_j-x|_\infty \le \frac{2^{d-1}}{2n+1}.
 \end{align*}
\end{theorem}
\begin{proof}
 If we assume on the contrary that $\delta>2^{d-1}/(2n+1)$, then there exist some point $x^*$ with $|x_j-x^*|_\infty>2^{d-1}/(2n+1)$ for all $j$.
 By translation invariance, we assume without loss of generality that $x^*=0$. 
 
 Now we consider the trigonometric polynomial $g(x)=(D_n(x)- (n+\frac{1}{2})^d)F_n(x)$ of degree $2n$.
 Using \cref{lem:DFn} with $C=2^{d-1}$, we estimate
 \begin{align*}
  D_n(x_j)-\left(n+\frac{1}{2}\right)^d \le  \frac{(2n+1)^d}{2^d}- \left(n+\frac{1}{2}\right)^d = 0.
 \end{align*}
 Since $F_n(x)\ge 0$ and $w_j>0$ this also implies
 \begin{align*}
  \sum_j w_j g(x_j) \le 0.
 \end{align*}
 On the other hand, we have
 \begin{align*}
     \int_{\T^d} g(x)\d x
     = \int_{\T^d} D_n(x)F_n(x)\d x - \left(n+\frac{1}{2}\right)^d \int_{\T^d} F_n(x)\d x
     = (n+1)^d - \left(n+\frac{1}{2}\right)^d >0.
 \end{align*}
 This contradicts that the quadrature rule is exact for $\Pi_{2n}^d$.
\end{proof}

\begin{remark}\label{rem:Improv}
We would like to discuss the following improvements and variants:
\begin{enumerate}
  \item For $d=1$, \cref{rem:DFn}i) allows to improve \cref{thm:deltaT} to $\delta\le 0.61/(2n+1)$. An even better result is due to Reimer and Yudin \cite[Thm.~6.21]{Re03}:
 If a quadrature rule with positive weights is exact for $\Pi_{2n-1}$, then the covering radius fulfils
 \begin{align*}
     \delta\le \frac{1}{4n}.
 \end{align*}
 In their work, the used test function is a trigonometric polynomial of degree $2n-1$, given by
 \begin{align*}
  \frac{\left(\cos(2\pi n x\right)^2}{\cos 2\pi x - \cos \frac{\pi}{2n}}=g(x)
  \begin{cases}
   \ge 0, & |x|\le \frac{1}{4n},\\
   \le 0, & \frac{1}{4n}\le|x|\le\frac12,
  \end{cases}
 \end{align*}
 having vanishing integral due to orthogonality. The squared cosine is non-negative with double zeroes of which the two near the origin are made simple by the denominator, see also \cref{fig:1d}.
 This result is sharp since it is well known that $2n$ equispaced points and weights $w_j=1/(2n)$ yield a quadrature rule which is exact for $\Pi_{2n-1}$. The covering radius of these points is exactly $1/(4n)$.
 Finally note that we were not able to generalize this test function to $d> 1$ directly.
\item The Dirichlet kernel in \cref{thm:deltaT} might be replaced by a stronger localized kernel. 
By increasing the localization with increasing space dimension, one might be able to get rid of the exponential dependence of the constant on the space dimension $d$. A constant independent of the space dimension seems not reasonable to expect, see, e.g., \cite{CaElGoKe22,CoDoGo24}.
 \item Ehler and Gröchenig \cite[Thm.~2.2]{EhGr23} built upon results in \cite{BaChCoGiSeTr14,BrEhGr18} and we exemplify their argument for comparison for $d=1$.
 Here, the test function $f$ is smooth and compactly supported and we assume without loss of generality $x_j\not\in\supp f\subset [-\delta,\delta]$ such that $\sum_j w_j f(x_j)=0$.
 The reasonable test function
 \begin{align*}
     f(x)
     =\delta^2+\sum_{k\ne 0} \frac{\sin^2(\pi\delta k)}{\pi^2 k^2} \eip{kx}
     =\begin{cases} \delta-|x|,& |x|<\delta,\\
     0,&\text{otherwise},
     \end{cases}
 \end{align*}
 allows with $r_n(x)=\sum_{|k|>n} \frac{\sin^2(\pi\delta k)}{\pi^2 k^2} \eip{kx}$ for the representation
 \begin{align*}
   \delta^2
   &= \int_{\T} f(x)\d x - \sum_j w_j f(x_j)
   =\int_{\T} r_n(x)\d x - \sum_j w_j r_n(x_j)\\
   &= \sum_{|k|>n} \left(\sum_j w_j \eip{kx_j}\right) \frac{\sin^2(\pi\delta k)}
   {\pi^2 k^2}.
 \end{align*}
 The first appearing sum as well as the second integral vanish, the switch from $f$ to $r_n$ is due to the exactness of the quadrature.
 In total, this however only yields the suboptimal estimate $\delta^2\le C/n$ when bounding the sum over $j$ and the $\sin^2$-term each by one.
 An improvement to $\delta^2\le C/n^2$ (with some hard to access constant $C$) might be achieved by \cite[Thm.~2.12]{BaChCoGiSeTr14}.
 The results relies on a pointwise estimate \cite[Lemma 8]{BaChCoGiSeTr14}, which can be achieved via Abel summation in the above situation.
 A more involved part of the full proof of \cite[Thm.~2.12]{BaChCoGiSeTr14} is \cite[Lemma 2.9]{BaChCoGiSeTr14}.
 We note that this lemma already yields the wanted relation between covering radius and degree by setting $s<\delta$ and choosing $y$ such that the quadrature vanishes.
 Finally, the proof of this lemma heavily relies on approximation rates for bandlimited functions minorizing and majorizing characteric functions, cf.~\cite{CoGiTr11}. We note that the involved minorizing function might be considered sign-localized and \cref{thm:deltaT} uses just the sign change but no approximation rate for such a function.
\end{enumerate}
\end{remark}

\subsection{Variants on the interval, the sphere, and the hyperbolic cross}\label{sec:gen}
 The reasoning above easily translates to other settings. We exemplify this in the following
 \begin{itemize}
     \item for the interval $[-1,1]$ showing that the covering radius of nodes is $O(1/n)$ for any positive quadrature rule of exactness $n$,
     \item for the unit sphere, making the results of \cite{EhGr23} more explicit, and
     \item for the hyperbolic cross, where we upper bound the dispersion of quadrature nodes. 
 \end{itemize}
\begin{remark}[Algebraic polynomials]
 The Legendre polynomial $P_n\colon[-1,1]\to\R$ is given by
 \begin{align*}
     P_n(x)=\frac{1}{2^n n!}\frac{\d^n}{\d x^n}\left(\left(x^2-1\right)^n\right)
 \end{align*}
 and has simple zeros $t_\ell\in(-1,1)$, $\ell=1,\hdots,n$.
 Now consider the polynomials
 \begin{align*}
  g_\ell(x)=\frac{(P_n(x))^2}{(x-t_\ell)(x-t_{\ell+1})},\qquad \ell=1,\hdots,n-1,
 \end{align*}
 of degree $2n-2$, which by orthogonality and considering their double and simple zeros fulfil
 \begin{align}\label{eq:PnOrth}
     \int_{-1}^1 g_\ell(x) \d x = 0\quad\text{and}\quad
     g_\ell(x) \begin{cases}< 0,& x\in (t_\ell,t_{\ell+1}),\\
                            > 0,& x\in \R\setminus \left((t_\ell,t_{\ell+1})\cup \{t_\ell:\ell=1,\hdots,n\}\right),
                            \end{cases}
 \end{align}
 respectively.
 If a quadrature rule with positive weights 
 \begin{align}\label{eq:PnEx}
     \sum_{j} w_j f(x_j) = \int_{-1}^1 f(x)\d x
 \end{align}
 is exact of degree $2n-2$, then every closed interval $[t_\ell,t_{\ell+1}]$ must contain at least one quadrature points $x_j$, which can be seen as follows:
 First note that the number of quadrature points is at least $n-1$.
 Now, fix $\ell$ and \emph{assume} that $[t_\ell,t_{\ell+1}]$ does not contain any of the quadrature points $x_j$, then we have $g_\ell(x_j)>0$ for at least one quadrature point and hence
 $    \sum_{j} w_j g_\ell(x_j) > 0$
 contradicting $\sum_{j} w_j g_\ell(x_j)=0$ by \eqref{eq:PnOrth} and \eqref{eq:PnEx}.
 Moreover, expressing the zeros as $t_\ell=\cos\theta_\ell$ allows for the estimate, see, e.g., \cite{Sz36},
 \begin{align*}
     \frac{2\ell-1}{2n+1}\pi < \theta_\ell < \frac{2\ell}{2n+1}\pi.
 \end{align*}
 Together with the above argument and the mean value theorem for the arccos-function, whose derivative is lower bounded by 1, this yields
 \begin{align*}
     \sup_{x\in [-1,1]} \inf_j |x_j-x|
     \le \sup_{x\in [-1,1]} \inf_j|\arccos x_j - \arccos x| \le \frac{3\pi}{2n+1}
 \end{align*}
 for any quadrature rule with positive weights and degree of exactness $2n-2$.
\end{remark}

\begin{remark}[Explicit bound on the sphere]
Reimer and Yudin construct a zonal spherical polynomial 
which is positive around a fixed point on the sphere and turns non-positive at a specified distance.
This allows to prove the following, see \cite[Thm.~6.21]{Re03}:
If a quadrature rule on the unit sphere $\S^{d-1}=\{x\in\R^d:\|x\|_2=1\}$ is exact for polynomials of degree $2n$, then 
\begin{align*}
 \delta=\max_{x\in\S^{d-1}}\min_j \arccos\langle x_j,x\rangle \le \theta_{n,1}
 \le\frac{3d}{2n}
\end{align*}
where $\theta_{n,1}$ denotes the smallest positive root of the 'trigonometric' Gegenbauer polynomial $C_n^{(d-2)/2}(\cos\theta)$ and the last bound holds for $d\ge 1$ and $n\ge 3$ as straightforward simplification of \cite[Thm.~1.1]{Ni19}. This estimate remains true for curves on the sphere and allows to lower bound the length of a curve by a packing argument as in \cite[second part of the proof of Thm.~2.2]{EhGr23}.
Following the latter, we obtain $\ell(\gamma)\ge\delta \omega(\S^{d-1})/\omega(B_{4\delta})\ge c_d n^{d-2}$, where $B_{4\delta}=\{x\in\S^{d-1}:\arccos\langle e_d,x\rangle\le 4\delta\}$ is a spherical cap (around the north pole $e_d$) and the constant $c_d$ decays superexponentially with $d$.
The recent work \cite{EhGr26} gives the same order but a better (larger) constant $c_d$. We compare these two approaches on the torus in the subsequent \cref{sec:length}.
\end{remark}

\begin{remark}[Hyperbolic cross trigonometric polynomials and dispersion]
 For notational convenience let $d=2$. The decay estimate in \cref{lem:DFn} can be slightly generalized for $p,q\in\N$ to
 \begin{align*}
   D_p(x) D_q(y)\le \frac{(2p+1)(2q+1)}{4},\quad |x|\in \left[\frac{1}{2p+1},\frac12\right] \text{ or } |y|\in \left[\frac{1}{2q+1},\frac12\right].
 \end{align*}
 We consider the test functions
 \begin{align*}
     g_{p,q}(x,y)=\left(D_p(x) D_q(y) - \left(p+\frac12\right)\left(q+\frac12\right)\right)F_p(x)F_q(y),
 \end{align*}
 which fulfil
 \begin{align*}
  0&\ge g_{p,q}(x,y)\quad \text{for}\quad (x,y)\in \T^2\setminus
  \left[-\frac{1}{2p+1},\frac{1}{2p+1}\right]\times \left[-\frac{1}{2q+1},\frac{1}{2q+1}\right]  \quad \text{and}\\
  0&<(p+1)(q+1)-\left(p+\frac12\right)\left(q+\frac12\right) = \int_{\T^2} g_{p,q}(x,y) \d x \d y.
 \end{align*}
 Now the hyperbolic cross of order $n$ is defined by
 \begin{align*}
     H_n=\left\{(k_1,k_2)\in\Z^2:(1+|k_1|)(1+|k_2|) \le n+1\right\}
 \end{align*}
 and we call $f\colon\T^2\to\C$, $f(x)=\sum_{k\in H_n} c_k \eip{kx}$,
 a hyperbolic cross trigonometric polynomial of degree $n$. In particular, the test functions $g_{p,q}$ are hyperbolic cross trigonometric polynomials of order $4n+3$ whenever $(1+2p)(1+2q)\le 4n+4$.
 If a quadrature formula $Q[f]=\sum_j w_j f(x_j)$ with positive weights $w_j>0$ and node set $X=\{x_j\in\T^2:j=1,\hdots,m\}$ is exact for hyperbolic cross trigonometric polynomials of order $4n+3$, then necessarily the dispersion (maximal volume of empty box) fulfils
 \begin{align*}
     \disp(X)=\sup_{[a_1,a_2)\times[b_1,b_2)\cap X=\emptyset} (b_1-a_1)(b_2-a_2) \le \frac{3}{n}.
 \end{align*}
 Assuming the contrary implies together with translation invariance and boundedness of $\T^2\simeq [-\frac12,\frac12)^2$ that $[-a,a)\times[-b,b)\cap X=\emptyset$ with $4ab\ge 3/n$ for some $a,b\in(0,\frac{1}{2}]$. Then choosing
 \begin{align*}
     p=\left\lceil\frac{1}{2a}-\frac12\right\rceil,\qquad
     q=\left\lceil\frac{1}{2b}-\frac12\right\rceil
 \end{align*}
 implies
 $
  (1+2p)(1+2q)\le \left(\frac{1}{a}+2\right)\left(\frac{1}{b}+2\right)\le 4n+4
 $ 
 and $g_{p,q}(x,y)\le 0$ for $(x,y)\in\T^2\setminus[-a,a]\times[-b,b]$. Hence $Q[g_{p,q}]\le 0$, which contradicts the exactness of the quadrature rule.
\end{remark}

\subsection{Lower bounds on the length of a curve}\label{sec:length}

A standard packing argument allows to relate the covering radius and the length of a curve, see, e.g., \cite[last part of the proof of Thm.~2.2]{EhGr23}. We note in passing that this relation is also well known for the traveling salesman problem in the unit cube, see \cite{BaClDu24} for a recent survey, and that we are not aware of any relation of this type which does not involve an exponentially small constant with increasing space dimension.

\begin{lemma}\label{lem:deltaL}
 Let $d\in\N$ and $\gamma\colon[0,1]\to\T^d$ be a curve as above with covering radius $\delta=\delta_{\Gamma}\le 1/8$, then
 \begin{align*}
    \ell(\gamma)\ge 2^{1-3d} \delta^{1-d}.
 \end{align*}
\end{lemma}
\begin{proof}
 We set $m=\lfloor{1/(2\delta)}\rfloor$ and use the partition
 \begin{align*}
  [0,1)^d=\bigcup_{k\in\{0,\hdots,m-1\}^d} Q_k,\qquad Q_k=\frac1m \left(k + [0,1)^d\right).
 \end{align*}
 The curve must visit every closed cube $\overline{Q_k}$ by the definition of the covering radius. Now we consider only the even-numbered cubes $k=(k_1,\hdots,k_d)\in\N^d$, $k_1,\hdots,k_d\in2\N$. The Euclidean distance of two even-numbered cubes is at least $2\delta$ and thus
 \begin{align*}
  \ell(\gamma)
  \ge 2\delta\left\lfloor{\frac{m}{2}}\right\rfloor^d
  \ge 2\delta\left(\frac{1}{4\delta}-1\right)^d
  \ge 2\delta \left(\frac{1}{8\delta}\right)^d,
 \end{align*}
 where the last bound uses $\delta\le 1/8$.
\end{proof}

\begin{corollary}\label{cor:Tschak2}
 If a curve $\gamma\colon[0,1]\to\T^d$ with covering radius at most $\delta_{\Gamma}\le 1/8$ and a non-negative weight function $w\colon\Gamma\to[0,\infty)$ allows for exact integration for all $f\in\Pi_{2n}^d$, i.e.,
 \begin{align*}
  \frac{1}{\ell(\gamma)}\int_{\gamma} f(x) w(x) \d \sigma(x)=\int_{\T^d} f(x)\d x,
 \end{align*}
 then the length of the curve is bounded from below by
 \begin{align*}
     \ell(\gamma)\ge \frac{1}{2^{d(d+1)}} \cdot (2n+1)^{d-1}.
 \end{align*}
\end{corollary}
\begin{proof}
 Tschakaloff's theorem, see, e.g., \cite[Thm.~1.24]{Sc17}, ensures the existence of a finite set $X=\{x_j\in\Gamma:j=1,\hdots,m\}$ of points and positive weights $w_j>0$ such that
 \begin{align*}
   \frac{1}{\ell(\gamma)}\int_{\gamma} f(x) w(x) \d \sigma(x)
   =\sum_j w_j f(x_j)
 \end{align*}
 for all $f\in\Pi_{2n}^d$.
 Now of course, this does not decrease the covering radius, i.e., $\delta_X\ge\delta_{\Gamma}$.
 Applying \cref{thm:deltaT} and \cref{lem:deltaL} yields the result.
\end{proof}

Recently \cite{EhGr26}, the above reasoning in \cite[last part of the proof of Thm.~2.2]{EhGr23} has been greatly improved by avoiding the detour via the covering radius and we use these arguments to give the following lower bound on the length of a curve that integrates trigonometric polynomials exactly.

\begin{theorem}\label{thm:Improv}
 Under the conditions of \cref{cor:Tschak2}, we have
 \begin{align*}
     \ell(\gamma)
     \ge \frac{0.94}{\sqrt{2d+1}} \cdot (2n+1)^{d-1}.
 \end{align*}
\end{theorem}
\begin{proof}
 In the recent work \cite{EhGr26}, upper bounds on the quadrature-measure of small sets as well as lower bounds on the length of the curves allowing an $L^2$-Marcinkiewicz-Zygmund (MZ) inequality are proven.
 We exemplify their main arguments for curves $\Gamma\subset\T^d$: If a quadrature is exact of degree $2n$ and following \cite[Sect.~4.1]{EhGr26}, then in particular 
 \begin{align*}
     \frac{1}{\ell(\gamma)}\int_{\gamma} \mathds{1}_{S}(x) w(x)\d\sigma(x) \cdot \min_{x\in S}|f(x)|^2 \le 
     \frac{1}{\ell(\gamma)}\int_{\gamma}|f(x)|^2 w(x) \d\sigma(x)=\|f\|_2^2
 \end{align*}
 holds for all $\sigma$-measurable sets $S\cap\Gamma\subset \Gamma$ and $f\in\Pi_n^d$ (where $\mathds{1}_{S}$ denotes the indicator funtion of the set $S$).
 For simplicity and comparability, we assume $\delta_{\Gamma}<1/8$. Choosing $x_0\in\Gamma$ and $\nu\in(0,\frac12]$, the small box $S=S(x_0)=x_0+\frac{1}{2n+1}[-\nu,\nu]^d$ does not contain the curve completely, $\Gamma\not\subset S$.
 Hence, the curve enters the box $S$, passes through its center $x_0$, and leaves the box $S$ again. By the pigeonhole principle/probabilistic method, this yields at least one point $x_0\in\Gamma$ such that also in the weighted case
 \begin{align*}
  \int_{\gamma} \mathds{1}_{S}(x) w(x)\d\sigma(x)\ge \frac{2\nu}{2n+1}.   
 \end{align*}
 Using the Dirichlet kernel as test function and assume by translation invariance once more $x_0=0$, we get
 \begin{align*}
     \ell(\gamma)
     &\ge
     \frac{\min_{x\in S}|D_n(x)|^2}{\|D_n\|_2^2}\int_{\gamma} \mathds{1}_{S}(x)\d\sigma(x)
     \ge
     \frac{2\nu}{(2n+1)^{d+1}} \left|\frac{\sin\pi\nu}{\sin\frac{\pi \nu}{2n+1}}\right|^{2d}\\
     &\ge (2n+1)^{d-1}\cdot 2\nu(1-\pi^2\nu^2/6)^d
 \end{align*}
 where the Dirichlet kernel attains its minimum in the 'corner' $x=\frac{\nu}{2n+1}(1,\hdots,1)$ and we use $x\ge\sin x\ge x(1-x^2/6)$.
 The last factor gets maximal for $\nu= \sqrt{6/\pi^2 (2 d + 1)}$ leading to
 \begin{align*}
     \ell(\gamma)
     \ge (2n+1)^{d-1}\cdot \frac{2\sqrt{6}\left(1-\frac{1}{2d+1}\right)^d}{\pi\sqrt{2d+1}}
     \ge \frac{0.94}{\sqrt{2d+1}} \cdot (2n+1)^{d-1}.
 \end{align*}
\end{proof}

\section{A simple curve and variants}\label{sec:simple}

Let $d,n\in\N$ and define the curve $\gamma=\gamma_{d,n}\colon[0,1]\to\T^d\simeq [0,1)^d$,
\begin{align}\label{eq:gamma}
    y\mapsto \gamma(y)=\gamma_{d,n}=\left(y,(2n+1)y,\hdots,(2n+1)^{d-1} y\right) \mod 1,
\end{align}
see \cref{fig:gamma} for an illustration. This might be viewed as a continuous analogue of a Korobov rank-1 lattice.
\begin{figure}[h!]
 \includegraphics[width=0.49\textwidth]{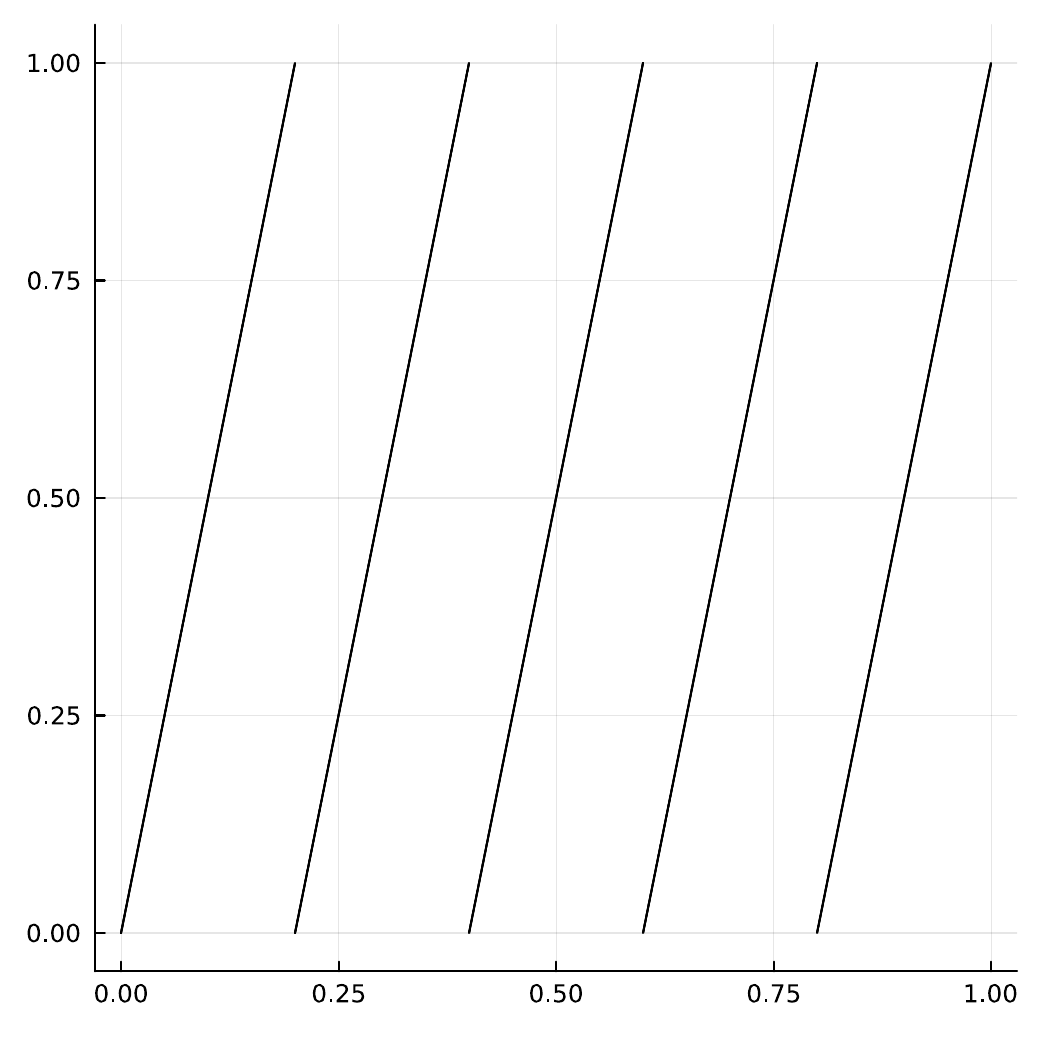}
 \includegraphics[width=0.49\textwidth]{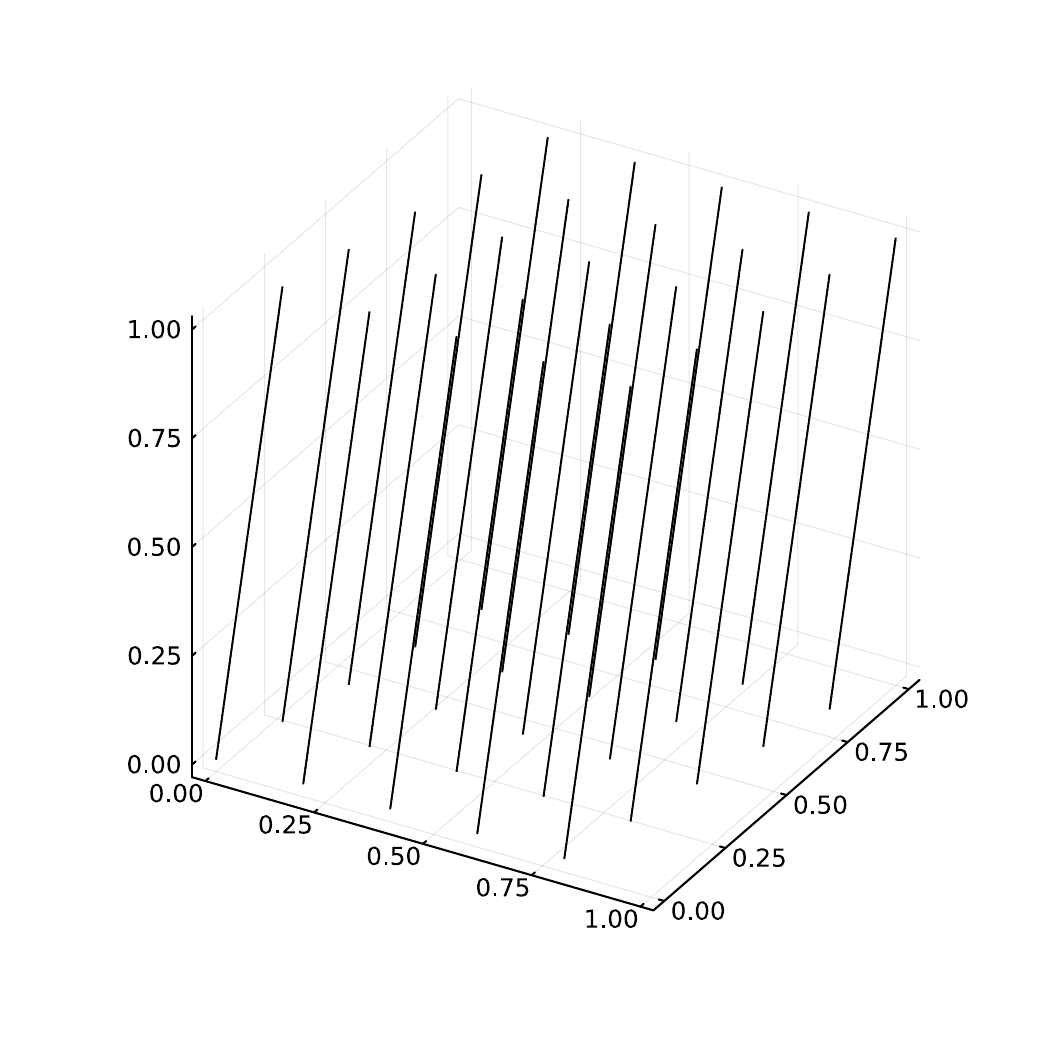}
 \caption{The curve \eqref{eq:gamma} for $n=2$ and $d=2$ (left) and $d=3$ (right).}\label{fig:gamma}
\end{figure}

\begin{theorem}\label{thm:gamma1}
 Let $d,n\in\N$, $f\in\Pi_n^d$, and $\gamma=\gamma_{d,n}$ be the curve \eqref{eq:gamma}, then
\begin{align}\label{eq:intEq}
    \frac{1}{\ell(\gamma)}\int_\gamma f(x)\d \sigma(x) = \int_{\T^d} f(x)\d x. 
\end{align}    
\end{theorem}
\begin{proof}
 We prove the assertion by integrating the standard basis of trigonometric polynomials. Obviously, we have for $e_0(x)=1$ the equality
 \begin{align*}
    \frac{1}{\ell(\gamma)}\int_\gamma 1 \d \sigma(x)
    =1= \int_{\T^d} 1 \d x.
\end{align*}
Now let $k\in \{-n,\hdots,n\}^d$, $k\ne 0$, then
\begin{align*}
    e_k(\gamma(y))
    =\exp\left(2\pi\ii\left(k_1 y + k_2 (2n+1) y +\hdots+k_d (2n+1)^{d-1} y \right)\right)
    = e_m(y)
\end{align*}
 with $m=k_1+(2n+1)k_2+\hdots+(2n+1)^{d-1} k_d\ne 0$ and hence
 \begin{align*}
    \frac{1}{\ell(\gamma)}\int_\gamma e_k (x) \d \sigma(x)
    =\int_0^1 e_k (\gamma(y)) \d y
    =\int_0^1 e_m(y) \d y
    =0 = \int_{\T^d} e_k(x) \d x 
\end{align*}
\end{proof}

\begin{remark}(Optimality with respect to degree, see \cite[Thm.~2]{BoDeVi17})
Let $F\colon\left(\C\setminus\{0\}\right)^d\to\C$,
\begin{align*}
 F(X)=\sum_{\|k\|_\infty\le n} c_k X^k=\sum_{k_1=-n}^n\dots\sum_{k_d=-n}^n c_{k_1,\hdots,k_d} X_1^{k_1}\cdot\hdots\cdot X_d^{k_d},
\end{align*}
denote a multivariate Laurent polynomial of degree $n$ in each variable, then $\tilde F\colon\C\to\C$,
\begin{align*}
 \tilde F(Y)=F\left(Y,Y^{2n+1},Y^{(2n+1)^2}\hdots Y^{(2n+1)^{d-1}}\right),
\end{align*}
is a univariate Laurent polynomial of degree $(2n+1)^{d-1} n$.
This reduction technique is sometimes denoted as Kronecker substitution.
Restricting to the complex unit circle, we get the multivariate and univariate trigonometric polynomials
\begin{align*}
 f(x)=f(x_1,\hdots,x_d)=F(\eip{x_1},\hdots,\eip{x_d}), \quad
 \tilde f(y)=\tilde F(\eip{y})
\end{align*}
of degree $n$ in each variable and of degree $(2n+1)^{d-1} n$, respectively.
We have $\dim \Pi_n^d=\dim \Pi_{(2n+1)^{d-1}n}^1=(2n+1)^d$ and hence 
\begin{align*}
    \tilde f=\arg\min_{\gamma} \deg f\circ \gamma,
\end{align*}
where the minimum is taken over all $\gamma$ such that \eqref{eq:intEq} holds (and the degree is infinity by convention if $f\circ\gamma$ is no trigonometric polynomial).
\end{remark}

The curve \eqref{eq:gamma} is also quasi-optimal with respect to covering radius and length as the following discussion shows.
\begin{lemma}\label{thm:gamma2}
 Let $d,n\in\N$, $d\ge 2$, then the curve \eqref{eq:gamma} has length
 \begin{align*}
    (2n+1)^{d-1}\le
    \ell(\gamma)
    =\left(\frac{(2n+1)^{2d}-1}{(2n+1)^2-1}\right)^{\frac12}
    \le \frac{3}{\sqrt{8}} (2n+1)^{d-1}
\end{align*}
and covering radius
\begin{align*}
    \frac{1}{4n+4} \le \delta = \frac{(2n+1)^{d-2}}{2(2n+1)^{d-1}+2} \le \frac{1}{4n+2}.
\end{align*}
\end{lemma}
\begin{proof}
We compute the length of this curve by
\begin{align*}
    \ell(\gamma)
    =\int_{\gamma}\d\sigma(x)
    =\int_0^1 \|\dot\gamma(y)\|_2\d y
    =\left(\sum_{r=0}^{d-1} (2n+1)^{2r}\right)^{\frac12}
    =\left(\frac{(2n+1)^{2d}-1}{(2n+1)^2-1}\right)^{\frac12}.
\end{align*}
The lower bound follows by considering the last summand, $r=d-1$, in the next to last term.
The upper bound is obtained via
\begin{align*}
 \frac{(2n+1)^{2d}-1}{(2n+1)^2-1}
 =(2n+1)^{2d-2}\left(\frac{(2n+1)^{2}-(2n+1)^{2-d}}{(2n+1)^2-1}\right)
 \le (2n+1)^{2d-2}\left(\frac{(2n+1)^{2}}{(2n+1)^2-1}\right)
\end{align*}
and the very last term is decreasing with $n$ and bounded by $\frac{9}{8}$ for $n\ge 1$.

Regarding the covering radius, we start by proving the upper bound. The points
\begin{align}\label{eq:lattice}
 \gamma\left(\frac{k}{(2n+1)^{d}}\right),\qquad k=0,\hdots,(2n+1)^{d}-1,
\end{align}
form a rank-1 lattice in $\T^{d}$. Placing a cube (square for $d=2$) with radius $\frac{1}{2(2n+1)}$ at every point gives a partition of $\T^{d}$. The covering radius of the curve is upper bounded by the covering radius of this subset of points.

The same upper bound can be obtained, by considering the hyperplane $\T^{d-1}\times\{0\}$: The points on the curve with last coordinate being zero, i.e.,
\begin{align*}
 \gamma\left(\frac{k}{(2n+1)^{d-1}}\right),\qquad k=0,\hdots,(2n+1)^{d-1}-1,
\end{align*}
form the rank-1 lattice \eqref{eq:lattice} with $d$ replaced with $d-1$ in $\T^{d-1}\times\{0\}$.
Hence, placing a $(d-1)$-cube (line segment for $d=2$, square for $d=3$) with radius $\frac{1}{2(2n+1)}$ at every point gives a partition of $\T^{d-1}$. Now moving these $(d-1)$-dimensional cubes along the curve also covers $\T^d$.

During this movement, the $d$-cube placed at the origin also covers points in the hyperplane $\T^{d-1}\times\{0\}$ with $\ell^\infty$-distance larger than $\frac{1}{2(2n+1)}$. We replace the $d$-cube of radius $\frac{1}{2(2n+1)}$ by one with smaller radius $\delta$.
When moving the center point of the cube along the curve up to the last coordinate being $\delta$, this cube still intersects with the 
hyperplane $\T^{d-1}\times\{0\}$ and the first coordinate (as well as all others) are increased by (at least) $\delta/(2n+1)^{d-1}$. Hence, the covering radius fulfils
\begin{align*}
 \frac{1}{2(2n+1)}=\delta+\frac{1}{(2n+1)^{d-1}}\delta.
\end{align*}
Solving for $\delta$ yields the result.
\end{proof}

\begin{corollary}[Quasi-Optimality]
 Let $n,d\in\N$ and $n$ be even.
 The curve \eqref{eq:gamma} is quasi-optimal in the sense that it integrates $f\in\Pi_n^d$ exactly and its exact geometric quantities fulfil 
 \begin{align*}
  \delta&\ge\frac{1}{4}\cdot\frac{1}{n+1}
  &&\text{and}&
  \ell(\gamma)\le \frac{3}{\sqrt{8}}\cdot(2n+1)^{d-1}
 \intertext{and thus are within a dimension dependent multiple of the necessary bounds}
 \delta&\le2^{d-1}\cdot\frac{1}{n+1}
  &&\text{and}&
  \ell(\gamma)
     \ge \frac{0.94}{\sqrt{2d+1}} \cdot (n+1)^{d-1}.
 \end{align*}
\end{corollary}
\begin{proof}
 See \cref{thm:gamma2}, \cref{thm:deltaT}, and \cref{thm:Improv} (the latter two for $\Pi_n^d$).
\end{proof}

\begin{example}
 For $d=2$, \cref{thm:gamma2} and direct calculation yields
 \begin{align*}
     \delta&=\frac14\cdot\frac{1}{n+1}
     &&\text{and}&
     \ell(\gamma)
     &=\sqrt{(2n+1)^2+1}
     \le 2\cdot (n+1).
 \intertext{If $n$ is even, the upper bound from \cref{thm:deltaT} gives $\delta\le 2/(n+1)$, which can be improved to $\delta\le 1/(n+1)$ by \cref{rem:DFn} ii).
 Finally, ignoring the very last estimate in the proof of \cref{thm:Improv} yields the following lower bound on the length:}
     \delta&\le \frac{1}{n+1}
     &&\text{and}&
     \ell(\gamma)
     &\ge \frac{32\sqrt{6}}{25\pi\sqrt{5}}\cdot(n+1)\ge 0.44\cdot(n+1).
 \end{align*}
\end{example}

\subsection{A minor generalization}
For notational convenience, we consider the two-dimensional case $d=2$.
We slightly generalize \eqref{eq:gamma} for coprime $p,q\in\N$ to
$\gamma\colon[0,1]\to\T^2\simeq [0,1)^2$,
\begin{align*}
    t\mapsto \gamma(t)=\left(pt,qt\right) \mod 1.
\end{align*}
Along the lines of the previous section, we have $\ell(\gamma)=\sqrt{p^2+q^2}$ and
 $e_k(\gamma(t))=\exp(2\pi i (k_1 p t +k_2 q t)=e_m(t)$ with $m=k_1p+k_2q$. 
 In particular, we have $m\ne 0$ for $|k_1|<q$ and $|k_2|<p$ unless $k=0$ and thus
\begin{align}\label{eq:gammapq}
 \frac{1}{\ell(\gamma)}\int_{\gamma} e_k \d\sigma=\begin{cases}
 1 & k=0,\\
 0 & k\ne 0,\,|k_1|<q,\,|k_2|<p.
 \end{cases}
\end{align}
Implicitly, this curve has been used quite heavily in different applications. When discretized to the lattice points $\gamma(\frac{j}{pq})$, $j=0,\hdots,pq-1$, it reduces a one-dimensional discrete Fourier transform of size $pq$ to a two-dimensional discrete Fourier transform of size $p\times q$, cf.~\cite{Go58}.

\subsection{A variant on the two-sphere}\label{sec:sphere}
 The standard parameterization of the sphere is given by $\gamma_2\colon[0,\pi]\times[0,2\pi)\to\S^2$, 
 $(\theta,\varphi)\mapsto(\sin\theta\cos\varphi,\sin\theta\sin\varphi,\cos\theta)$,
 and integrating on the sphere reads as
 \begin{align*}
     \int_{\S^2} f(x) \d\omega(x)
     =\int_0^{2\pi} \int_0^\pi f(\theta,\varphi) \sin\theta \d\theta\d\varphi.
 \end{align*}

 Let $p,q\in\N$, $2p,q$ coprime, and consider the curve $\gamma_1\colon[0,1]\to [0,\pi]\times[0,2\pi)$,
\begin{align}\label{eq:gammaS2}
    t\mapsto \gamma_1(t)=\left(\theta(t),\varphi(t)\right)=\left(2\pi pt \mod \pi,2\pi qt \mod 2\pi\right).
\end{align}

\begin{figure}[h!]
 \includegraphics[width=0.32\textwidth]{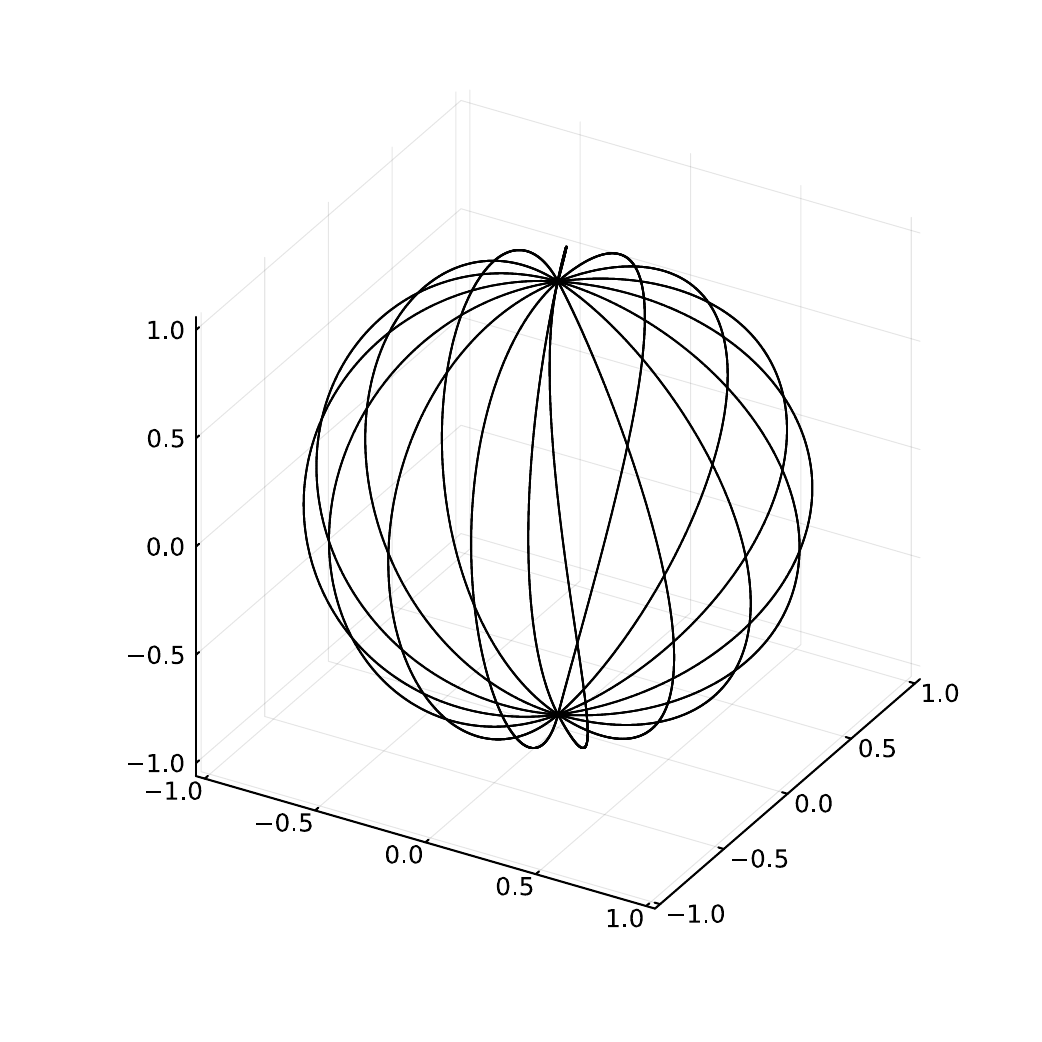}
 \includegraphics[width=0.32\textwidth]{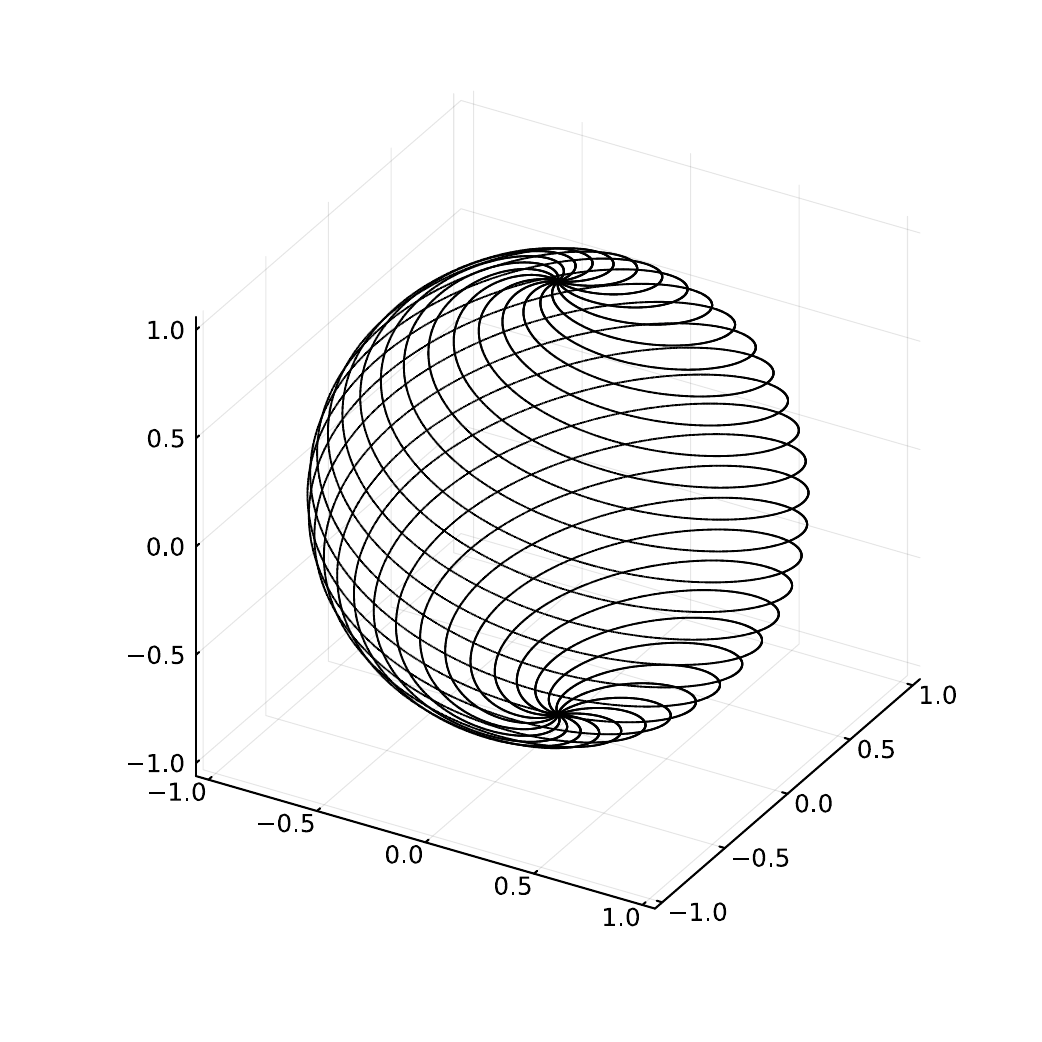}
 \includegraphics[width=0.32\textwidth]{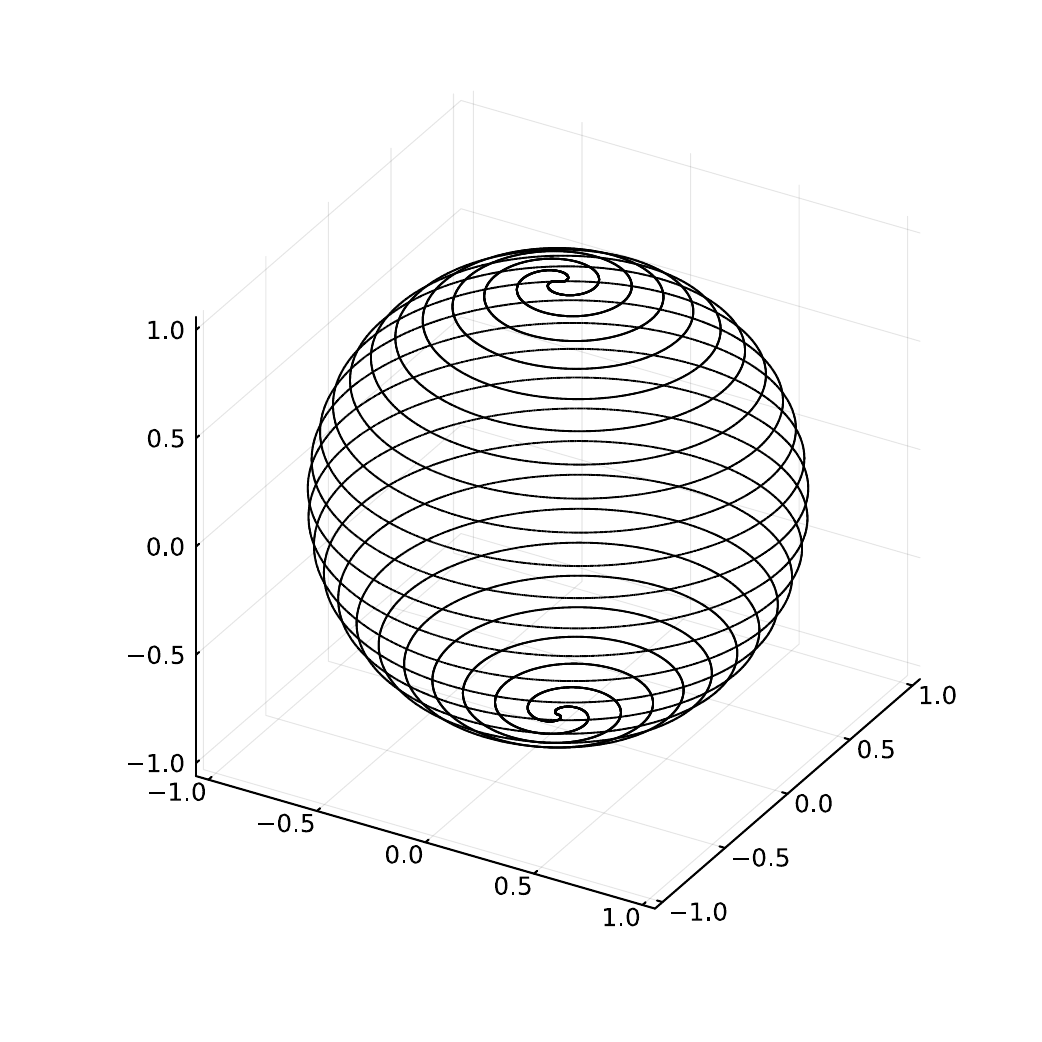}
 \caption{Curves with $p=8$, $q=1$ (left), $p=8$, $q=23$ (middle), and $p=1$, $q=23$ (right).}\label{fig:Sd}
\end{figure}

 Now consider the composition $\gamma:[0,1]\to\S^2$, $\gamma=\gamma_2\circ\gamma_1$, which can be viewed as $2p$ curve segments on the sphere by definition (jumping from the south to the north pole by $\mod \pi$), $p$ closed smooth curves by 'glueing' together, e.g., $t\searrow 0$ and $t\nearrow \frac12$, or one closed curve with $2p$ kinks.
 Direct computation yields
 \begin{align*}
     \ell(\gamma)
     =\int_\gamma \d \sigma(x)
     =\int_0^1 \|\dot\gamma(t)\|_2\d t
     =2\pi\int_0^1 \sqrt{p^2+q^2\sin^2(2\pi p t)}\d t
     \le 2\pi\sqrt{p^2+q^2},
 \end{align*}
 where the next to last term is a so-called elliptic integral.
 It turns out useful, to include a weight by defining the measure in cartesian coordinates $x=(x_1,x_2,x_3)\in\S^2$ by
 \begin{align*}
  \d\tilde\sigma(x)=w(\arccos{x_3})\d\sigma(x),\qquad
  \tilde w(\theta):=\sqrt{\frac{(p^2+q^2)\sin^2\theta}{p^2+q^2\sin^2\theta}}.
 \end{align*}
 The weight function is bounded by $\sin\theta\le \tilde w(\theta)\le \sin\theta\sqrt{1+q^2/p^2}$ where the lower bound is attained for $\theta=\pi/2$ and the upper bound is attained for $\theta=0$ and $\theta=\pi$.
 This new measure is normalized to
 \begin{align*}
  C_{\gamma}=\int_{\gamma}\d\tilde\sigma(x)=2\pi\int_0^1 \sqrt{(p^2+q^2)\sin^2(2\pi pt)}\d t
  =4\sqrt{p^2+q^2}.
 \end{align*}

\begin{theorem}
 We denote the standard orthogonal basis of spherical harmonics in colatitude-longitude parameterization by $Y^k_l\colon [0,\pi]\times[0,2\pi) \to\C$,
 \begin{align}\label{eq:Ykl}
     Y^k_l(\theta,\varphi)=\e^{i k\varphi} P^k_l(\cos\theta),\qquad l\in\NZ,k=-l,\hdots,l,
 \end{align}
 where $P^k_l$ denotes the so-called associated Legendre-functions. 
 Now, let $n\in \NZ$ and $f=\sum_{l=0}^n\sum_{k=-l}^l \hat f^k_l Y^k_l$ be a spherical harmonic polynomial of degree at most $n$.
 For $p,q\in\N$, $2p,q$ coprime and $p>n$, we have
    \begin{align*}
     \frac{1}{4\pi}\int_{\S^2} f(x) \d\omega(x)
     =\frac{1}{C_{\gamma}}\int_\gamma f(x) \d\tilde\sigma(x)
 \end{align*}
 with the curve and measure defined above.
\end{theorem}
\begin{proof}
  We prove the assertion for the basis of spherical harmonics.
  For $k=l=0$, we have
 \begin{align*}
     \int_\gamma Y^0_0(x) \d\tilde\sigma(x)
     =\int_\gamma \d\tilde\sigma(x)
     =C_\gamma
     \quad\text{and}\quad
     \int_{\S^2} \d\omega(x)=
     \int_0^{2\pi} \int_0^\pi \sin\theta \d\theta\d\varphi=4\pi.
 \end{align*}
 For $k=0$, $l\ne 0$, we argue by orthogonality of the Legendre polynomials and have
 \begin{align*}
     \int_\gamma Y^0_l(x) \d\tilde\sigma(x)
    &=2\pi\sqrt{p^2+q^2}
      \int_0^1 P_l(\cos(2\pi p t \mod \pi)) \sin(2\pi p t \mod \pi) \d t\\
    &=2\pi\sqrt{p^2+q^2} \sum_{r=0}^{2p-1}
      \int_{\frac{r}{2p}}^{\frac{r+1}{2p}} P_l(\cos(2\pi p t \mod \pi)) \sin(2\pi p t \mod \pi) \d t\\
    &=2\pi\sqrt{p^2+q^2} \sum_{r=0}^{2p-1}
        2\pi p \int_0^\pi P_l(\cos x) \sin x \d x=0.
 \end{align*}

 For $k\ne 0$, we argue by discrete orthogonality of the exponential functions, abbreviate $h_l(t)=P_l(\cos(2\pi p t \mod \pi)) \sin(2\pi p t \mod \pi)$, note that $h_l(t+r/2p)=h_l(t)$, $r\in\Z$, and thus have
\begin{align*}
     \int_\gamma Y^k_l(x) \d\tilde\sigma(x)
    &=2\pi\sqrt{p^2+q^2} \sum_{r=0}^{2p-1}
      \int_{\frac{r}{2p}}^{\frac{r+1}{2p}} h_l(t) \e^{2\pi ikqt} \d t\\
    &=2\pi\sqrt{p^2+q^2} \int_{0}^{\frac{1}{2p}} h_l(t) \e^{2\pi ikqt} \sum_{r=0}^{2p-1} \e^{2\pi ikq  \frac{r}{2p}} \d t=0
 \end{align*}
 where the sum in the last line is zero since $2p$ and $q$ are coprime and $|k|<p$.
\end{proof}

\begin{remark}
 According to \cite[Thm.~5.1 and Cor.~5.2]{EhGr26}, a generalized quadrature being exact of degree $2n$, $n\in\N$, $n\ge 2$, implies $\ell(\gamma)\ge n/2$.
 The curve \eqref{eq:gammaS2} with parameters $q=3$ and prime $p\in (2n,4n)$ yields
 \begin{align*}
     \ell(\gamma)\le 2\pi\sqrt{p^2+q^2}\le 8\pi n,
 \end{align*}
 making it optimal up to a factor. This also gives a simple proof for the \emph{weighted} version of \cite[Thm.~1.2]{EhGr23}.
\end{remark}

\section*{Acknowledgement}
Fruitful discussions with Matthias Reitzner and Marcin Wnuk on geometric aspects, in particular on \cref{lem:deltaL}, have been highly appreciated.

\bibliographystyle{abbrvurl}
\bibliography{refs}

\end{document}